\documentclass[12pt]{amsart}
\usepackage[utf8]{inputenc}
\usepackage[T2A]{fontenc}
\usepackage{fullpage}
\usepackage[english]{babel}

\usepackage{amsmath}
\usepackage{amssymb}
\usepackage{amsthm}
\usepackage{stmaryrd}
\usepackage[all]{xy}

\CompileMatrices

\theoremstyle{plain}
\newtheorem{thm}{Theorem}
\newtheorem{stmt}{Proposition}
\newtheorem{lemma}{Lemma}

\newtheorem{hyp}{Conjecture}

\theoremstyle{definition}
\newtheorem{definition}{Definition}
\newtheorem{remark}{Remark}

\def \prooflike #1 {{\sc {#1.}} }
\def \endproof {$\Box$}

\DeclareMathOperator{\End}{End}
\DeclareMathOperator{\Id}{Id}
\DeclareMathOperator{\rk}{rk}
\DeclareMathOperator{\diag}{diag}
\DeclareMathOperator{\Lie}{Lie}
\DeclareMathOperator{\SL}{SL}

\DeclareMathOperator{\Spin}{Spin}

\DeclareMathOperator{\Sp}{Sp}
\DeclareMathOperator{\F}{F}
\DeclareMathOperator{\E}{E}
\DeclareMathOperator{\G}{G}

\def \isom {\cong}
\def \WW {\mathcal{W}}
\def \NN {\mathbb{N}}
\def \ZZ {\mathbb{Z}}
\def \QQ {\mathbb{Q}}
\def \RR {\mathbb{R}}

\def \kk {\mathbb{K}}
\def \HV #1{X  \left( #1 \right)}

\def \LieG {\mathfrak{g}}
\def \LieH {\mathfrak{h}}
\def \LieT {\mathfrak{t}}

\begin{document}

\subjclass[2000]{13A50, 14L24, 14R20, 22E46}
\title{On stability of diagonal actions and tensor invariants}
\address{Department of Higher Algebra, Faculty of Mechanics and Mathematics, Lomonosov Moscow State University, Leninskie Gory 1, GSP-1, Moscow 119991, Russia}
\email{aanisimov@inbox.ru}
\author{Artem~B.~Anisimov}
\maketitle

\begin{abstract}
For a connected simply connected semisimple algebraic group~$G$ we prove existence of invariant tensors in tensor powers of rational $G$-modules and establish relations between existence of such invariant tensors and stability of diagonal actions of~$G$ on affine algebraic varieties.
\end{abstract}

\section{Introduction}

Recall that an action of a reductive algebraic group~$G$ on an affine variety~$X$ is called stable \cite{PopovStabilityOnFactorialManifold} if its generic orbits are closed. Many actions~$G : X$ do not have this property. In~\cite{ArzhStabilisationOfOperation} it was proved, however, that if~$G$ is semisimple, then every action~$G : X$ can be made stable by considering diagonal action of~$G$ on a sufficiently large number of copies of~$X$. Let us consider an example of diagonal action of~$SL_n\left( \kk \right)$ on product of~$k$ copies of~$\kk^n$. For small values of~$k$ such action is not stable, because there exists a dense $SL_n$-orbit in~$\left( \kk^n \right)^k$ which does not coincide with~$\left( \kk^n \right)^k$. For~$k = n$ generic orbits of this action are level surfaces of the determinant
\begin{equation*}
O_c = \left\{ \left( v_1, \dots, v_n \right) \in \kk^n \times \dots \times \kk^n \ | \ \det\left(v_1, \dots, v_n\right) = c \right\},
\end{equation*}
and are therefore closed. For~$k > n$ generic orbits are closed, too.

Stability of diagonal actions is closely related to existence of nonzero $G$-invariant elements in tensor powers of rational $G$-modules. Let us consider the standard representation of~$SL_n$ from this point of view. Representations in tensor powers $\SL_n : \left( \kk^n \right)^{\otimes k}$ with~$k < n$ have no nonzero invariant elements, while the action on the~$n$-th tensor power does have nonzero invariants. In this example we observe that the minimal tensor power that contains nonzero invariants is the same as minimal number of copies of~$\kk^n$ needed to obtain a stable action. This fact is no coincidence~---~as we will show later, absence of invariants in low tensor powers implies existence of diagonal actions with small number of copies.

Relations between stability of actions and tensor invariants have been revealed in~\cite[Theo\-rem 10]{VinbergStabilityOfRestrictedAction} and have later been used in~\cite{ArzhStabilisationOfOperation} to prove that every effective action of a semisimple group can be made stable by passing to an appropriate diagonal action.

In this article we continue investigation of relations between stability of diagonal actions and existence of nonzero invariant elements in tensor powers of rational modules. We provide lower and upper bounds of number of copies needed to obtain a stable diagonal action and explicitly calculate diagonal of the weight semigroup of action~$G : G^n$. These results extend the results of~\cite{ArzhStabilisationOfOperation} and prove that number of copies required to obtain a stable action depends only on group~$G$. A connection is also established between existence of tensor invariants and existence of \textit{balanced} collections of elements of Weyl groups.

Let us pass to formulation of main results. Below the ground field~$\kk$ is assumed to be an algebraically closed field of characteristic~zero; when no explicit characterisation of a group~$G$ is given, it is assumed to be connected simply connected and semisimple; weights of group~$G$ are taken with respect to a fixed maximal torus~$T \subseteq G$; simple roots and fundamental weights are numbered in the same way as in~\cite{VinbergOnischik}.

\begin{definition}
Let~$G$ be a connected algebraic group. Denote
\begin{equation*}
M\left( G \right) := \left\{ n \in \NN\ | \left( V^{\otimes n} \right)^G \neq \{0\}\ \mbox{for every nonzero rational}\ G\mbox{-module}\ V  \right\}.
\end{equation*}
Denote~$m(G)$ the minimal element of the semigroup~$M(G)$ or~$+\infty$, if~$M(G)$ is empty.
\end{definition}

One does not have to verify that $\left( V^{\otimes n} \right)^G \neq \{0\}$ for every nonzery $G$-module~$V$~---~it suffices to prove that all irreducible modules have this property. Indeed, if $n$-th tensor powers of all nonzero $G$-modules have nontrivial invariants then, à fortiori, $n$-th tensor powers of all irreducible $G$-modules have nonzero invariants. Conversely, fix a~$G$-module~$V$ and its irreducible submodule $U \subseteq V$; we have $\left(V^{\otimes n}\right)^G \supseteq \left(U^{\otimes n}\right)^G \neq \{0\}$.

\begin{thm}\label{M-of-simple-groups}
Semigroups~$M(G)$ with~$G$ simple are listed in the table below:
\begin{center}
\begin{tabular}{|ccccccc|}
\hline
$G$								&	\vline	&	$M(G)$	&	\vline	&	$G$					&	\vline	&	M(G)															\\
\hline
$\SL_n$						&	\vline	&	$n\NN$	&	\vline	& $\G_2$				&	\vline	&	$\left\{ n \in \NN\ |\ n \geq 2 \right\}$			\\
$\Spin_{2n+1}$			&	\vline	&	$2\NN$	&	\vline	& $\F_4$				&	\vline	&	$\left\{ n \in \NN\ |\ n \geq 2 \right\}$			\\
$\Spin_{4n+2}$			&	\vline	&	$4\NN$	&	\vline	& $\E_6$				&	\vline	&	$3\NN$														\\
$\Spin_{4n+4}$			&	\vline	&	$2\NN$	&	\vline	& $\E_7$				&	\vline	&	$2\NN$														\\
$\Sp_{2n}$					&	\vline	&	$2\NN$	&	\vline	& $\E_8$				&	\vline	&	$\left\{ n \in \NN\ |\ n \geq 2 \right\}$			\\
\hline
\end{tabular}
\end{center}
\end{thm}

Calculation of~$M(G)$ for an arbitrary (not necessarily reductive) group~$G$ can be reduced to the cases listed in the table above by applying the following two propositions.

\begin{stmt}\label{same-as-reductive}
Let~$G$ be a connected affine algebraic group, $F$ its unipotent radical and  $H = G \big/ F$. Then $M(G) = M(H)$.
\end{stmt}
In fact this proposition shows that~$M(G)$ is to be calculated only for semisimple groups~$G$, not for reductive groups. Indeed, if $Z \subseteq G$ is a nontrivial central torus in group~$G$, then~$M(G) = \emptyset$; it follows from the fact that such group~$G$ can act nontrivially by multiplications on~$\kk^1$ and on all tensor powers of~$\kk^1$.
\begin{stmt}\label{M-of-product}
Let $G = G_1 \times G_2$ be a product of two reductive groups $G_1$ and $G_2$. Then $M(G) = M\left( G_1 \right) \cap M\left( G_2 \right)$.
\end{stmt}

Applying this proposition one can easily find~$M(G)$ if~$G$ is a connected simply connected semisimple group, that is, if~$G$ is a product of simply connected simple groups. Considering groups~$G$ that are not simply connected is a more involved problem and it seems probable that every group~$G$ that is not simply connected requires an ad hoc approach. However, it is clear that if~$G_1$ and~$G_2$ are semisimple groups of the same type and~$G_1$ is simply connected, then~$M\left( G_1 \right) \subseteq M\left( G_2 \right)$.

It turns out that calculation of semigroups~$M(G)$ is tightly related to describing balanced collections of elements of the Weyl group of~$G$.

\begin{definition}
Let~$\WW$ be the Weyl group of~$G$. A collection of elements $w_1, \dots, w_k \in \WW$ is called \textit{balanced}, if $w_1 + \dots + w_n = 0$ (the sum is considered as a sum of endomorphisms of the $\QQ$-linear span of roots of $G$).
\end{definition}

\begin{thm}\label{existence-of-balanced-tuples}
Let~$G$ be a simple group and~$\WW$ be the Weyl group of~$G$. There exists a balanced collection of~$m$ elements of $\WW$ iff $m \in M(G)$.
\end{thm}

Now we pass to relations between semigroups~$M(G)$ and stibility of diagonal actions of groups~$G$.

\begin{definition}
Let~$G$ be a connected semisimple algebraic group which is not necessarily simply connected. Denote
\begin{itemize}
\item $s_m(G)$ the smallest natural number such that for every affine variety~$X$ with an effective action of~$G$ the diagonal action on product of~$s_m(G)$ copies of~$X$ is stable,
\item $s_s(G)$ the smallest natural number such that for every affine variety~$X$ with an effective action of~$G$ and for every~$k \geq s_s(G)$ the diagonal action of~$G$ on product of~$k$ copies of~$X$ is stable.
\end{itemize}
The numbers~$s_m(G)$ and~$s_s(G)$ are called \textit{metastability index} and \textit{stability index} respectively.
\end{definition}

Existence of number~$s_s(G)$ for a semisimple group~$G$ is stated by Theorem~\ref{s_s-upper-bound}. Obviously, we have~$s_m(G) \leq s_s(G)$.

The reason for separating metastability and stability indices is that stability of diagonal action on~$k$ copies of variety~$X$ does not imply stability of action on~$r$ copies of~$X$ with~$r > k$. Such phenomenon is exhibited by symplectic groups~$\Sp_{2m}$. Indeed, take~$X$ to be the standard representation of~$\Sp_{2m}$ in~$\kk^{2m}$ and take~$k \leq 2m$. It is easy to see that if~$k$ is even then the stabiliser in general position of~$\Sp_{2m} : \left( \kk^{2m} \right)^k$ is isomorphic to~$\Sp_{2m-k}$ and therefore reductive; by~\cite[Theorem 1]{PopovStabilityOnFactorialManifold} we have that the action~$\Sp_{2m} : \left( \kk^{2m} \right)^k$ is stable. If~$k$ is odd then s.~g.~p. of~$\Sp_{2m} : \left( \kk^{2m} \right)^k$ contains a nontrivial normal unipotent subgroup hence this action is not stable.

The following two statements give bounds of stability indices in terms of~$M(G)$.
\begin{thm}\label{s_m-lower-bound}
Let~$G$ be a simple simply connected group. Then~$m(G) \leq s_m(G)$.
\end{thm}

\begin{thm}\label{s_s-upper-bound}
Let~$e(G)$ be the smallest natural number such that for every affine variety~$X$ with effective action of~$G$ the diagonal action of~$G$ on~$e(G)$ copies of~$X$ has finite s.~g.~p. Then~$s_s(G) \leq e(G)m(G)$.
\end{thm}

Theorem~\ref{s_s-upper-bound} is proved by a simple modification of argument in~\cite[Theorem 1]{ArzhStabilisationOfOperation}. Note that number~$e(G)$ exists and is not greater than dimension of group~$G$.

The result of Theorem~\ref{s_m-lower-bound} can be substantially improved for groups that have only self-conjugate linear representations. This improvement can be made by applying results of~\cite{PanyushevSelfConjugate}.
\begin{thm}\label{s_m-for-self-conjugate}
Let~$G$ be a connected semisimple algebraic group which is not necessarily simply connected. Suppose additionally that all linear representations of~$G$ are self-conjugate. Then~$s_m(G) = m(G) = 2$.
\end{thm}

The author has considered several examples of actions of groups~$G$ that have linear representations which are not self-conjugate. These examples suggest that if~$G$ is simple then it is superfluous to suppose that all linear representations of~$G$ are self-conjugate.
\begin{hyp}
If~$G$ is a simple group then~$s_m(G) = m(G)$.
\end{hyp}

The author would like to thank I.~V.~Arzhantsev for stating the problem and for many helpful discussions. The idea of applying PRV-theorem to calculation of semigroups~$M(G)$ is due to D.~A.~Timashev. The author would also like to thank V.~L.~Popov for his valuable comments.

\section{Calculation of semigroups $M(G)$}\label{calculation-of-M}

\subsection{Auxiliary statements}

Demonstration of Theorem~\ref{M-of-simple-groups} relies on PRV-theorem on extre\-mal weights of submodules in tensor product of irreducible modules. Let us recall necessary defini\-tions and facts.

Denote~$\WW$ the Weyl group of~$G$ and let~$V(\lambda)$ be the irreducible~$G$-module with highest weight~$\lambda$. Let~$\tau$ be a weight occurring in~$V(\lambda)$. The weight~$\tau$ is said to be \textit{extremal} if it is~$\WW$-equivalent to~$\lambda$. Since every weight is~$\WW$-equivalent to a unique dominant weight, the module~$V(\lambda)$ is uniquely determined by any of its extremal weights. This observation permits us to define~$V(\tau)$ with~$\tau$ not necessarily dominant. The following statement is called PRV-theorem; it partially describes the decomposition of tensor product of two irreducible modules.

\begin{thm}{\rm (\cite{PRV}, \cite{MathieuPRV})}
Let~$\lambda$ and~$\mu$ be arbitrary weights. Then the tensor product~$V(\lambda) \otimes V(\mu)$ contains the irreducible submodule~$V(\lambda + \mu)$.
\end{thm}

PRV-theorem establishes the following relation between lengths of balanced collections in~$\WW$ and elements of~$M(G)$.
\begin{lemma}\label{length-of-balanced-tuple-is-in-M}
Let~$w_1, \dots, w_m \in \WW$ be a balanced collection of~$m$ elements. Then $M(G) \supseteq m\NN$.
\end{lemma}
\begin{proof}
Take any dominant weight~$\lambda$ and the irreducible module~$V(\lambda)$ which corresponds to it. We have $V\left( \lambda \right)^{ \otimes m} = V\left(w_1\lambda\right) \otimes V\left(w_2\lambda\right) \otimes \dots \otimes V\left(w_m\lambda\right)$. It follows from PRV-theorem that this module contains the submodule with extremal weight
\begin{equation*}
w_1\lambda + w_2\lambda + \dots + w_m\lambda = \left(w_1 + w_2 + \dots + w_n \right)\lambda= 0.
\end{equation*}
Hence we have $\left( V\left(\lambda\right)^{\otimes m} \right)^G \neq \left\{0\right\}$ and $M(G) \supseteq m\NN$.
\end{proof}

The above lemma proves one of implications of Theorem~\ref{existence-of-balanced-tuples}. The other implication, namely existence of balanced collections of~$m$ elements with~$m \in M(G)$ will be derived from proof of Theorem~\ref{M-of-simple-groups}.

The following statement is in most of the cases sufficient to prove that a given number~$m$ does not belong to~$M(G)$.
\begin{lemma}\label{center-restricts-M}
Let~$Z(G)$ be the center of~$G$ and let~$H \subseteq Z(G)$ be a cyclic subgroup of order~$m$ Then~$M(G) \subseteq m\NN$.
\end{lemma}
\begin{proof}
The group~$G$ has a faithful irreducible representation, therefore there exists a simple~$G$-module~$U$ such that~$H$ is faithfully represented in~$U$. The module~$U$ is irreducible with respect to~$G$ and the action of~$H$ commutes with that of~$G$. Therefore~$H$ acts by multiplica\-tions by powers of a $m$-th root of unity. Faithfulness of representation of~$H$ implies that one of its generators~$x_0$ acts by multiplication by a $m$-th root of unity; denote this root~$\varepsilon$. In every tensor power~$U^{\otimes k}$ the generator~$x_0$ acts by multiplication by~$\varepsilon^k$. Therefore if~$k$ is not divisable by~$m$ then~$H$ acts in~$U^{\otimes k}$ by nontrivial multiplications and~$U^{\otimes k}$ has no~$G$-invariant elements. It implies that~$M(G) \subseteq m\NN$.
\end{proof}

While proving Theorem~\ref{M-of-simple-groups} we will construct balanced collections in Weyl groups. Their construction in cases of Weyl groups of types~~$\F_4$, $\E_6$ и~$\E_8$ relies heavily on properties of Coxeter elements of these Weyl groups. Let us recall the definition of Coxeter element. Let~$\WW$ be the Weyl group corresponding to an irreducible essential root system~$\Phi$. The product of reflections corresponding to all simple roots in~$\Psi$ is called a \textit{Coxeter element} of~$\WW$. This definition depends on ordering of simple reflections, but all elements obtained in such way are conjugate in~$\WW$; therefore they all have the same order and the same eigenvalues. Later on by Coxeter element we mean any Coxeter element of~$\WW$.
\begin{thm}\label{order-of-coxeter-element}\cite[Proposition~3.18~и~Theorem~3.19]{Humphreys}
Let~$\Phi$ be an irreducible essential root system. Then order of its Coxeter element is~$h = \left| \Phi \right| \big/ \rk \Phi$.

Let~$r$ be the rank of~$\Phi$ and~$\exp \left( 2 \pi i m_1 / h \right)$, $ \dots $, $\exp \left( 2 \pi i m_r / h \right)$ be all eigenvalues of Coxeter element of~$\WW$ ($0 \leq m_i < h$). Then order of the Weyl group~$\left| \WW \right|$ equals $\prod\limits_i \left( m_i + 1 \right)$.
\end{thm}

The numbers~$m_i$ defined in the theorem above are called \textit{exponents} of the Weyl group~$\WW$. In cases that we consider the exponents can be calculated by applying the following statement.

\begin{lemma}\label{good-m-is-exponent}\cite[Proposition~3.20]{Humphreys}
Let~$\Phi$ be an irreducible essential root system, let $h$ be order of its Coxeter element and~$m$ be any natural number that is not greater than~$h$. Suppose additionally that~$m$ and~$h$ are coprime. Then~$m$ is one of exponents of the Weyl group corresponding to~$\Phi$.
\end{lemma}

In many cases the following statement can be used to prove that specific powers of Coxeter elements make up a balanced collection.
\begin{lemma}\label{transform-of-order-3}
Let~$\WW$ be the Weyl group of~$G$. Suppose that an element~$w \in \WW$ has order~$3$ and that~$1$ is not eigenvalue of~$w$. Then~$\left\{ w, w^2, w^3 \right\}$ is a balanced collection in~$\WW$. Furthermore~$M(G) \supseteq 3\NN$.
\end{lemma}
\begin{proof}
Note that for every ~$x \in \RR^{\rk G}$ the element~$\left( \Id + w + w^2 \right)x$ is $w$-invariant and therefore zero. That is why $\Id + w + w^2 = 0$ and $M(G) \supseteq 3\NN$.
\end{proof}

\subsection{Calculation of $M(G)$ for simple groups $G$}
\

\prooflike{Proof of Theorem~\ref{M-of-simple-groups}} \textit{Case 1:~$G = \SL_n$}. Let~$e_i$ be the vectors of the standard basis of~$\RR^n$. Simple roots of the system~$A_{n-1}$ are the vectors~$e_1 - e_2$, $e_2 - e_3$, $\dots$, $e_{n-1}-e_n$, the Weyl group of~$A_{n-1}$ is the symmetric group~$S_n$ and it acts in $\RR^n$ by permuting the coordinates. Denote~$\varepsilon \in \WW$ the cyclic permutation~$\left( 123 \dots n \right)$. We have
\begin{equation*}
\varepsilon + \varepsilon^2 + \dots + \varepsilon^n = \begin{pmatrix}
1			&	1			&	\dots	&	1			\\
1			&	1			&	\dots	&	1			\\
\vdots	&	\vdots	&	\ddots	&	\vdots	\\
1			&	1			&	\dots	&	1			\\
\end{pmatrix},
\end{equation*}
(the above sum is considered as a sum in $\End\left(\RR^n\right)$).

The restriction of this operator to the span of simple roots is zero. Indeed, the span of simple roots is the subspace~$\left\{ x_1 + x_2 + \dots + x_n = 0 \right\}$ and all such vectors are taken to zero by~$\varepsilon + \varepsilon^2 + \dots + \varepsilon^n$. By Lemma~\ref{length-of-balanced-tuple-is-in-M} we have $M\left( \SL_n \right) \supseteq n\NN$. The centre of~$\SL_n$ is isomorphic to the group of $n$-th roots of unity, thus by applying Lemma~\ref{center-restricts-M} we get the reverse inclusion $M\left( \SL_n \right) \subseteq n\NN$.

It is important to remark that the cyclic permutation~$\varepsilon$ is Coxeter element of the root system~$A_{n-1}$.

\textit{Case 2:~$G = \Spin_{2n+1}$ or~$G = \Sp_{2n}$}. In these two cases the Weyl group is~$S_n \rightthreetimes \left\{ \pm 1 \right\}^n$ and it acts in~$\RR^n$ by permuting the coordinates and changing signs of coordinates. This means that~$-\Id \in \WW$ and by Lemma~\ref{length-of-balanced-tuple-is-in-M} we have the inclusions $M\left( \Spin_{2n+1} \right),\ M\left( \Sp_{2n} \right) \supseteq 2\NN$. Both~$\Spin_{2n+1}$ and~$\Sp_{2n}$ have centres isomorphic to~$\ZZ_2 \left( = \ZZ / 2\ZZ \right)$~\cite[Table~3]{VinbergOnischik} and Lemma~\ref{center-restricts-M} yields the reverse inclusions~$M\left( \Spin_{2n+1} \right),\ M\left( \Sp_{2n} \right) \subseteq 2\NN$.

\textit{Case 3:~$G = \Spin_{2n}$}. The Weyl group of the root system~$D_n$ is~$S_n \rightthreetimes \left\{ \pm 1 \right\}^{n-1}$ and acts in~$\RR^n$ by permuting the coordinates and changing signs of the coordinates in even number of positions. It is necessary to consider two subcases.

If~$n$ is even then~$-\Id \in \WW$ and centre~$Z\left( \Spin_{2n} \right) \isom \ZZ_2 \oplus \ZZ_2$. This yields~$M\left( \Spin_{2n} \right) = 2\NN$.

Now suppose~$n$ is odd. In this case~$Z\left( \Spin_{2n} \right) \isom \ZZ_4$ and we have $M\left( \Spin_{2n} \right) \subseteq 4\NN$. Consider the following four elements of~$\WW$:
\begin{equation*}\begin{matrix}
w_1	&	=	&	\diag(	&	1,		&	-1,	&	-1,	&	-1,	&	\dots,	&	-1	&	),	\\
w_2	&	=	&	\diag(	&	-1,	&	1,		&	-1,	&	-1,	&	\dots,	&	-1	&	),	\\
w_3	&	=	&	\diag(	&	-1,	&	-1,	&	1,		&	1,		&	\dots,	&	1	&	),	\\
w_4	&	=	&	\diag(	&	1,		&	1,		&	1,		&	1,		&	\dots,	&	1	&	).	\\
\end{matrix}
\end{equation*}
These elements add up to zero hence~$M\left( \Spin_{2n} \right) \supseteq 4\NN$.

\textit{Case 4:~$G = \G_2$}. In this case the Weyl group is the dihedral group~$\mathcal{D}_6$ of order~$12$. We have~$-\Id \in \WW$ and $M \left( \G_2 \right) \supseteq 2\NN$. Let~$\varepsilon \in \WW$ be the rotation by~$2\pi / 3$. We have~$\Id + \varepsilon + \varepsilon^2 = 0$ hence $M\left( \G_2 \right) \supseteq 3\NN$. As a result we get $M\left( \G_2 \right) = \left\{ n \in \NN\ |\ n \geq 2 \right\}$.

\textit{Case 5:~$G = \F_4$}. The Weyl group~$\WW$ corresponding to the group~$F_4$ contains~$-\Id$~\cite[Table 1]{VinbergOnischik} hence~$M\left( \F_4 \right) \supseteq 2\NN$. Let~$\varepsilon$ be Coxeter element of~$\WW$. According to Theorem~\ref{order-of-coxeter-element} the element~$\varepsilon$ has order~12 and, according to Lemma~\ref{good-m-is-exponent}, it has~$1, 5, 7, 11$ for exponents. As a result, the element~$\varepsilon^4$ has no real eigenvalues. Applying Lemma~\ref{transform-of-order-3} to~$\varepsilon^4$ we obtain the inclusion $M\left( \G_2 \right) \supseteq 3\NN$. As we can see, $M\left( \F_4 \right) = \left\{ n \in \NN\ |\ n \geq 2 \right\}$.

\textit{Case 6:~$G = \E_6$}. Consider Coxeter element~$\varepsilon$. It has order~12. Unlike the previous case Lemma~\ref{good-m-is-exponent} yields only four exponents~1, 5, 7 and 11. Eigenvalues of Coxeter elements come in pairs~$\lambda$ and~$\bar \lambda$, therefore the remaining two exponents are~$m$ and~$12-m$. Theorem~\ref{order-of-coxeter-element} states that order of the Weyl group~$\left| \WW \right| = 2^7 \cdot 3^4 \cdot 5$ coincides with product $2 \cdot 12 \cdot 6 \cdot 8 \cdot (m+1) \cdot (12-m+1)$. From this equality we find the remaining exponents. They are~4 and~8. Thus the element~$\varepsilon^4$ has no real eigenvalues and Lemma~\ref{transform-of-order-3} yields the inclusion~$M\left( \E_6 \right) \supseteq 3\NN$. Centre $Z\left( \E_6 \right)$ is~$\ZZ_3$ hence $M\left( \E_6 \right) = 3\NN$.

\textit{Case 7:~$G = \E_7$}. The Weyl group corresponding to the group~$E_7$ contains the mapping~$-\Id$. Therefore $M\left( \E_7 \right) \supseteq 2\NN$. Centre~$Z\left( \E_7 \right)$ is~$\ZZ_2$ hence $M\left( \E_7 \right) = 2\NN$.

\textit{Suppose 8:~$G = \E_8$}. The Weyl group~$\WW$ corresponding to the group~$E_8$ contains the mapping~$-\Id$. Therefore $M\left( \E_8 \right) \supseteq 2\NN$. Let~$\varepsilon \in \WW$ be Coxeter element. Its order is~30 and Lemma~\ref{good-m-is-exponent} yields its eights exponents which are coprime with~30. As a result the element~$\varepsilon^{10}$ has no real eigenvalues hence Lemma~\ref{transform-of-order-3} is applicable to it. In this way we obtain the inclusion~$M(\E_8) \supseteq 3\NN$ and it proves that $M\left( \E_8 \right) = \left\{ n \in \NN\ |\ n \geq 2 \right\}$.
\endproof

\prooflike{Proof of Theorem~\ref{existence-of-balanced-tuples}}
In view of Lemma~\ref{length-of-balanced-tuple-is-in-M} it remains to prove that if~$m \in M(G)$ then there exists a balanced collection containing~$m$ elements. Such balanced collections have been constructed in the above proof.
\endproof

\subsection{Calculation of $M(G)$ for an arbitrary group $G$}
Let us first prove Proposition~\ref{same-as-reductive} which asserts that semigroups~$M(G)$ need to be calculated only for reductive groups.

\begin{proof}
Inclusion~$M(G) \subseteq M(H)$ is obvous. Indeed, denote~$\pi : G \rightarrow H$ the natural map. Every~$H$-module~$V$ can be considered as a~$G$-module with multiplica\-tion~$g \cdot x = \pi(g)x$ and we have $\left( V^{\otimes m} \right)^H = \left( V^{\otimes m} \right)^G$.

Take~$m \in M(H)$ and a~$G$-module~$V$. Its submodule~$W = V^F$ is nonzero. The unipotent radical~$F$ is a normal subgroup in~$G$, thus the action $G : V$ gives rise to action~$G  : W$. From this statement it follows that~$W$ is also $H$-module; since~$m \in M(H)$ we have $\left( W^{\otimes m} \right)^H \neq \left\{0\right\}$. So we have $\left( V^{\otimes m} \right)^G \supseteq \left( W^{\otimes m} \right)^G = \left( W^{\otimes m} \right)^{G / F} \neq \left\{0\right\}$ and $M(G) \supseteq M(H)$.
\end{proof}

It has already been remarked that a reductive group~$G$ with nontrivial central torus has empty semigroup~$M(G)$. It suffices therefore to calculate~$M(G)$ for semisimple groups~$G$. If~$G$ is semisimple and simply connected then it is a product of several simply connected simple groups and Proposition~\ref{M-of-product} yields~$M(G)$.

\prooflike{Proof of Proposition~\ref{M-of-product}}
Take~$m \in M\left( G_1 \times G_2 \right)$. Let~$V$ and~$W$ be arbitrary modules over~$G_1$ and~$G_2$ respectively. Each of them can be considered as a module over~$G_1 \times G_2$ with trivial action of one of the factors. By choice of~$m$ we have $\left( V^{\otimes m} \right)^{G_1} = \left( V^{\otimes m} \right)^{G_1 \times G_2} \neq \{0\}$ and $\left( W^{\otimes m} \right)^{G_2} = \left( W^{\otimes m} \right)^{G_1 \times G_2} \neq \{0\}$. Thus $m \in M\left( G_1 \right)$ and $m \in M\left( G_2 \right)$ and we obtain the inclusion $M\left( G_1 \times G_2 \right) \subseteq M\left( G_1 \right) \cap M\left( G_2 \right)$.

Conversely, take~$m \in M\left( G_1 \right) \cap M\left( G_2 \right)$ and let~$U$ be an irreducible module over~$G_1 \times G_2$. Both groups~$G_1$ and~$G_2$ are reductive hence~$U = V \otimes W$ for appropriate irreducible modules~$V$ and~$W$ over~$G_1$ and~$G_2$ respectively. We have $U^{\otimes m} \isom V^{\otimes m} \otimes W^{\otimes m}$. In view of choice of~$m$ we have $\left( V^{\otimes m} \right)^{G_1} \neq \{0\}$ and $\left( W^{\otimes m} \right)^{G_2} \neq \{0\}$. As a result $\left( V^{\otimes m} \otimes W^{\otimes m} \right)^{G_1 \times G_2} \neq \{0\}$. It proves that $m \in M\left( G_1 \times G_2 \right)$.
\endproof

\section{Relation between stability indices and $m(G)$}

\subsection{Auxiliary facts about HV-varieties}

In order to prove Theorem~\ref{s_m-lower-bound} we need to give examples of actions $G : X$ such that diagonal actions $G : X^{m(G)-1}$ are not stable. Necessary examples are given by actions on so-called HV-varieties. All facts that we need about these varieties can be found in~\cite{HVmanifolds} and~\cite{Popov}.

Let~$\lambda$ be a dominant weight of~$G$ and~$v_\lambda$ be the highest weight vector in~$V(\lambda)$. Consider the~$G$-orbit of ~$v_\lambda$. Its closure is called a \textit{HV-variety} corresponding to the dominant weight~$\lambda$~\cite[Definition 1]{HVmanifolds}.
\begin{thm}\cite[Theorem 1]{HVmanifolds}
Let~$\lambda$ be a dominant weight of~$G$ and~$v_\lambda$ be the highest weight vector in~$V(\lambda)$. Then $X(\lambda) = G \cdot v_\lambda \cup \left\{0\right\}$.
\end{thm}

A collection $\left( \lambda_1, \dots, \lambda_s \right)$ of dominant weights of~$G$ is said to be \textit{invariant-free}~\cite[Defini\-tion~2]{Popov} if $\left( V\left(n_1\lambda_1\right) \otimes \dots \otimes V\left(n_s\lambda_s\right) \right)^G = \left\{0\right\}$ for every tuple of natural numbers $n_1, \dots, n_s$.
\begin{thm}\cite[Theorem 10]{Popov}\label{inv-free-implies-instability}
Let $\left( \lambda_1, \dots, \lambda_s \right)$ be a collection of dominant weights of~$G$. The following properties are equivalent:
\begin{itemize}
\item the collection $\left( \lambda_1, \dots, \lambda_s \right)$ is invariant-free,
\item the closure of every~$G$-orbit in $\HV{\lambda_1} \times \dots \times \HV{\lambda_s}$ contains $0 \in V\left(\lambda_1\right) \oplus \dots \oplus V\left(\lambda_s\right)$,
\item $\kk\left[\HV{\lambda_1} \times \dots \times \HV{\lambda_s}\right]^G = \kk$.
\end{itemize}
\end{thm}

\subsection{Auxiliary facts about tensor products of~$\Spin_{2r}$-modules}

In order to prove Theorem~\ref{s_m-lower-bound} for~$G = \Spin_{4n+2}$ we need to find explicitly the decomposition of a certain tensor product. To this end we employ the generalised Littlewood-Richardson rule. Necessary facts about this generalisation can be found in~\cite{LR}.

\begin{definition}\cite[Appendix A.3]{LR}
Let~$\varpi_i, 1 \leq i \leq r$ be the fundamental weights of~$\Spin_{2r}$ and let~$\lambda = \sum\nolimits_{i=1}^{r}{a_i \varpi_i}, a_i \geq 0$ be its dominant weight. A Young diagram of shape~$\lambda$ is a Young diagram corresponding to the partition~$\left( c_1, \dots, c_r \right)$ with~$c_i$ defined as:
\begin{equation*}
c_p = \left\{ \begin{array}{ll}
2\sum\limits_{i=p}^{r-2}{a_i} + a_{r-1} + a_r	&	\mbox{if}\ p \leq r-2, 	\\
a_{r-1} + a_r															&	\mbox{if}\ p = r-1,\\
a_r																			&	\mbox{if}\ p = r.
\end{array}\right.
\end{equation*}
\end{definition}

\begin{remark}
We treat the numbers~$c_i$ as lengths of rows ($c_1$ being the length of the bottom row) and draw the rows left-aligned and from bottom to top.
\end{remark}

\begin{definition}\cite[Appendix A.4]{LR}
Let~$T$ be the Young diagram of shape~$a\varpi_{2r}$ and suppose that its cells are filled with natural numbers. The diagram~$T$ with filled cells is said to be a~$\Spin_{2r}$-standard Yound tableau if it satisfies the following requirements:
\begin{itemize}
\item all cells of~$T$ contain natural numbers that are not greater than~$2r$;
\item entries in rows are strictly ascending (the rows are oriented left-to-right);
\item entries in columns are ascending (the columns are oriented bottom-to-top);
\item no row contains~$i$ and~$2r+1-i$ simultaneously;
\item every row has even number of entries that are greater than~$r$.
\end{itemize}
\end{definition}

\begin{remark}
The definition of standard Young tableau~$T$ of arbitrary shape~$\mu$ is more involved and imposes more constraints on entries of~$T$. We will not provide this definition in full detail for it is of no use to our later arguments. An intereseted reader is encouraged to consult~\cite[Appendix A.4]{LR} and see the definition in its full generality.
\end{remark}

\begin{definition}\cite[Appendix A.4]{LR}
Let~$T$ be a standard Young tableau. Denote~$C_T(i)$ the number of entries of~$T$ that are equal to~$i$. Define a weight of tableau~$T$ as
\begin{equation*}
v(T) = \frac{1}{2} \left[ \left( C_T(1) - C_T(2r) \right)\varepsilon_1 + \left( C_T(2) - C_T(2r-1) \right)\varepsilon_2 + \dots \right].
\end{equation*}
Denote~$v_m(T)$ the weight of tableau~$T_m$ obtained from~$T$ by removing all rows below the $m$-th one.
\end{definition}

\begin{definition}\cite[Appendix A.4]{LR}
Let~$\mu$ be a dominant weight. A standard Young tableau~$T$ is called~$\mu$-dominant if the weights~$2\mu + 2v_m(T)$ are dominant for every~$m$.
\end{definition}

\begin{thm}\cite[Appendix A.4]{LR}
Let~$\lambda$ and~$\mu$ be dominant weights of~$\Spin_{2r}$. Then
\begin{equation*}
V(\lambda) \otimes V(\mu) = \bigoplus\limits_T V(\lambda + v(T)),
\end{equation*}
the sum on the right-hand side runs over all~$\lambda$-dominant standard Young tableaux of shape~$\mu$.
\end{thm}

\subsection{An example of an invariant-free triple of weights of the group~$G = \Spin_{4n+2}$}

In ~\cite{Popov} it has been proved that the collection $\left( \varpi_{2n+1}, \varpi_{2n+1}, \varpi_{2n+1} \right)$ of weights of~$\Spin_{4n+2}$ is \textit{primitive}, that is $\dim \left( V\left( n_1\varpi_{2n+1} \right) \otimes V\left( n_2\varpi_{2n+1} \right) \otimes V\left( n_3\varpi_{2n+1} \right) \right)^{\Spin_{4n+2}} \leq 1$ for all natural numbers $n_1, n_2, n_3$. We need more accurate information about this collection. Precisely, we need to prove that it is invariant-free.

\begin{lemma}
Let~$p$ and~$q$ be two natural numbers such that~$p \geq q$. Then we have the following decomposition:
\begin{equation*}
V\left(p\varpi_{2n+1}\right) \otimes V\left(q\varpi_{2n+1}\right) =
{\bigoplus} V\left( (p+q-2r)\varpi_{2n+1} + \varpi_{i_1} + \dots + \varpi_{i_r} \right),
\end{equation*}
the sum on the right-hand side runs over all~$r$ in~$0, \dots, q$ and over all collections of odd natural numbers $1 \leq i_1 \leq i_2 \leq \dots \leq i_r \leq 2n-1$.
\end{lemma}
\begin{proof}
A standard Young tableau~$T$ of shape~$r\varpi_{2n+1}$ is a rectangle with~$2n+1$ columns and~$r$ rows. Since a row of~$T$ has~$2n+1$ entries, it is uniquely defined by those of its entries that are not greater than~$2n+1$. Let~$I = \left\{ i_1 < i_2 < \dots < i_p \right\}$ and~$J = \left\{ j_1 < \dots < j_s \right\}$ be two sets of natural numbers such that~$I \sqcup J = \left\{ 1, \dots, 2n+1 \right\}$. If a row of~$T$ starts with~$I$ then the remaining numbers are necessarily~$4n+3-j_s, \dots, 4n+3-j_1$. The weight of such row is~$\left( \sum\nolimits_{i=1}^{2n+1}{a_i \varepsilon_i} \right) / 2$ with~$a_i = +1$ if~$i \in I$ and~$a_i = -1$ if~$i \in J$.

Denote for brevity~$i' = (4n+3-i)$. In what follows we say that elements of~$J$ are \textit{removed} from the interval~$1, \dots, (2n+1)$ and a row of tableau~$T$ that corresponds to~$I$ and~$J$ (that is, one that starts with~$I$) is said to be obtained from interval~$1, \dots, 2n+1$ by removing elements of~$J$.

Let us describe all~$\varpi_{2n+1}$-dominant Young tableaux~$T$. First consider two adjacent rows of~$T$. Let~$p$ and~$q$ be the smallest numbers removed from the top and bottom line respectively. Then these rows end with numbers~$p'$ and~$q'$ respectively. Tableau~$T$ is standard hence~$p' \geq q'$ and~$p \leq q$. Now let us show that every row of~$T$ is obtained by removing trailing numbers from the interval~$1, \dots, 2n+1$. Combining this with the previous statement we conclude that~$T$ looks like the tableau below ($k_1 \leq k_2 \leq \dots \leq k_{r}$):
\begin{equation*}
\begin{array}{|lcc|}
\hline
1 \dots k_1 \	\vline \ (2n+1)' \dots\dots\dots	&	&	(k_1+1)'	\\
\hline
1 \dots\dots k_2 \ \vline \ (2n+1)'\dots\dots &	&	(k_2+1)' \\
\hline
1\dots\dots\dots k_3 \ \vline \ (2n+1)'\dots &	&	(k_3+1)' \\
\hline
\hdotsfor{3} \\
\hline
1 \dots	&	&	2n+1 \\
\hline
\end{array}.
\end{equation*}
Note that all numbers~$k_i$ are odd because every row has even number of entries that are greater than~$2n+1$.

To this end consider the topmost row of the tableau~$T$. If it is obtained from~$1, \dots, 2n+1$ by removing any set other than a trailing interval~$s, \dots, 2n+1$ then there are two numbers $1 \leq x < y \leq 2n+1$ such that~$x$ is removed while~$y$ is not. It implies that~$2p\varpi_{2n+1} + 2v_1(T)$ equals $(p+1)\varepsilon_1 + \cdots + (p-1)\varepsilon_x + \cdots + (p+1)\varepsilon_y + \cdots $. Since the coefficient of~$\varepsilon_x$ is smaller than the coefficient of~$\varepsilon_y$ the weight~$2p\varpi_{2n+1} + 2v_1(T)$ is not dominant. This contradicts the assumption of~$p\varpi_{2n+1}$-dominance of the tableau~$T$. Therefore the topmost row of~$T$ is obtained from~$1, \dots, 2n+1$ by removing some trailing part of this interval.

Now we proceed by induction. Suppose that top~$l$ rows of the tableau~$T$ are obtained by removing trailing intervals. Let~$k_1+1, \dots, k_l+1$ be the smallest numbers removed from the top rows of~$T$. We assume inductively that we have inequalities~$k_1 \leq \dots \leq k_l$ and equality
\begin{equation*}
2p\varpi_{2n+1} + 2v_l(T) = 2(p+q-2l)\varpi_{2n+1} + 2\varpi_{k_1} + \dots 2\varpi_{k_l}.
\end{equation*}
Without loss of generality we may assume that~$k_l < 2n+1$. Let us apply to the~\hbox{$(l+1)$-st} row the same argument that we have applied to the topmost row of~$T$. The smallest number removed from the~$(l+1)$-st row is not smaller than the smallest number removed from the~\hbox{$l$-th} row. Therefore the numbers~$x$ and~$y$ yielded by the argument will be greater than~$k_l$. The fundamental weights~$\varpi_{k_i}$ are sums of~$\varepsilon_i$ with $i \leq k_l  < x$ hence they do not influence the coefficients of~$\varepsilon_x$ and of~$\varepsilon_y$. From this fact it follows that the reasoning based on comparison of coefficients of~$\varepsilon_x$ and of~$\varepsilon_y$ stays valid and proves that the weight~$2p\varpi_{2n+1} + 2v_{l+1}(T)$ is not dominant.

As we can see, every standard~$p\varpi_{2n+1}$-dominant Young tableau looks like one on the picture above. From the proof it follows that, conversely, every Young tableau depicted above is~$p\varpi_{2n+1}$-dominant if~$q \leq r$. Indeed, such tableau is standard and all partial weights~$2p\varpi_{2n+1} + 2v_l(T)$ are dominant because $p + q - 2l \geq p-q \geq 0$.

The weight of the tableau~$T$ depicted above is $(p+q-2s) + \varpi_{k_1} + \dots + \varpi_{k_s}$ with $s$ being the number of rows that have some of numbers removed and~$(k_i+1)$ being the smallest number removed in~$i$-th row. This means that all irreducible modules contained in the decomposition of~$V\left( p\varpi_{2n+1} \right) \otimes V\left( q\varpi_{2n+1} \right)$ equal $V\left( (p+q-2s)\varpi_{2n+1} + \varpi_{k_1} + \dots + \varpi_{k_s} \right)$ for appropriate collection of~$k_i$.

To complete the demonstration we have to show that every weight
\begin{equation*}
\lambda = (p+q-2l)\varpi_{2n+1} + \varpi_{k_1} + \cdots + \varpi_{k_l},
\end{equation*}
with~$k_i$ being odd natural numbers not greater than~$2n-1$ can be obtained as~$p\varpi_{2n+1} + v(T)$ for appropriately chosen ~$p\varpi_{2n+1}$-dominant Young tableau~$T$. Without loss of generality we may assume that~$k_1 \leq \dots \leq k_l$. Fill the topmost row of~$T$ in the following way: write numbers $1, \dots, k_1$ into~$k_1$ starting cells and pad them with~$(2n+1)', \dots, (k_1+1)'$. Next~$l-1$ rows are filled analogously and last~$q-l$ are filled with~$1, \dots, (2n+1)$. It is clear that for the tableau~$T$ constructed by this process we have~$\lambda = p\varpi_{2n+1}+v(T)$. Applying the rule of Littlewood-Richardson we get that~$V(\lambda)$ is indeed a submodule of~$V\left( p\varpi_{2n+1} \right) \otimes V\left( q\varpi_{2n+1} \right)$. It is obvious that~$\lambda$ can be uniquely represented as~$p\varpi_{2n+1} + v(T)$ hence~$V(\lambda)$ is contained in the tensor product with multiplicity~one.
\end{proof}

For brevity we will employ the multiindex notation. Let~$I = \left( i_1, \dots, i_s \right)$ be a multiindex with all components~$i_j$ being odd natural number not greater than~$2n-1$. Denote~$\left| I \right|$ the number of components of~$I$ and define~$\varpi_I$ as the sum~$\sum\nolimits_{i \in I}\varpi_i$. Using this notation one can rewrite the decomposition of~$V\left(p\varpi_{2n+1}\right) \otimes V\left(q\varpi_{2n+1}\right)$ in this way: $\bigoplus V\left( (p+q-2\left| I \right|)\varpi_{2n+1} + \varpi_I \right)$.

\begin{lemma}\label{half-spin-rep-is-inv-free}
The triple $(\varpi_{2n+1}, \varpi_{2n+1}, \varpi_{2n+1})$ is invariant-free.
\end{lemma}
\begin{proof}
Take tree natural numbers~$p \geq q \geq r$. According to the previous lemma we have
\begin{multline*}
V\left( p\varpi_{2n+1} \right) \otimes V\left( q\varpi_{2n+1} \right) \otimes V\left( r\varpi_{2n+1} \right) = \\
= \bigoplus\limits_{I} \left[ V\left( (p+q-2\left| I \right|)\varpi_{2n+1} + \varpi_I \right) \otimes V\left( r\varpi_{2n+1} \right) \right].
\end{multline*}

Let us show that no module in the right-hand side contains $\Spin_{4n+2}$-invariant elements. The tensor product $V\left( (p+q-2\left| I \right|)\varpi_{2n+1} + \varpi_I \right) \otimes V\left( r\varpi_{2n+1} \right)$ decomposes into direct sum of~$V\left( (p+q-2\left| I \right|)\varpi_{2n+1} + \varpi_I + v(T) \right)$ for approriately chosen Young tableaux~$T$. Let us show that $-\left( (p+q-2s)\varpi_{2n+1} + \varpi_{i_1} + \dots + \varpi_{i_s} \right)$ can not be equal to weight of any standard tableau~$T$. To this end, fix an arbitrary standard Young tableau~$T$ of shape~$r\varpi_{2n+1}$, that is, a rectangle with~$(2n+1)$ columns and~$r$ rows. The tableau~$T$ is standard and therefore its~$t$ bottom rows start with~$1$ and the other~$r-t$ rows start with numbers that are greater than~$1$, hence~$T$ has weight~$v(T) = 1/2 \left( (t - (r-t))\varepsilon_1 + \dots \right)$. Therefore we have
\begin{equation*}
(p+q-2\left|I\right|)\varpi_{2n+1} + \varpi_I + v(T) =
\left( \frac{p+q-r}{2} + t \right)\varepsilon_1 + \dots
\end{equation*}

If $(p+q-2\left|I\right|)\varpi_{2n+1} + \varpi_I + v(T) = 0$ then $(p+q-r)/2 + t = 0$. The last equality is absurd because~$(p+q-r)/2 \geq p/2 > 0$.
\end{proof}

\subsection{An example of diagonal not stable action of the group~$G = \E_6$ on two copies of variety~$X$}

\begin{lemma}
Let~$G$ be a semismple algebraic group,~$\LieG$ be its Lie algebra and ~$T$ be a maximal torus in~$G$. Let~$\LieG = \LieT \oplus \bigoplus\limits_{\alpha \in \Delta}{\LieG_\alpha}$ be the weight decomposition of~$\LieG$ with respect to~$T$. Finally, let~$V$ be a module over~$G$ and~$v \in V$ be a weight vector with respect to~$T$. Then the Lie algebra~$\LieH$ of stabiliser of~$v$ is regular, that is, it equals~$\LieH = \LieH_0 \oplus \bigoplus\limits_{\alpha \in \Gamma}\LieG_\alpha$ with~$\LieH_0 \subseteq \LieT$ and~$\Gamma \subset \Delta$.
\end{lemma}
\begin{proof}
Consider an arbitrary element~$\xi \in \LieH$. Let~$\xi = \sum\limits_{\mu \in \Delta}{\xi_\mu}$ be its weight decomposition. Since~$0 = \xi \cdot v = \sum\limits_{\mu \in \Delta}{\xi_\mu \cdot v}$ and every summand is a weight vector, we conclude that every summand is zero, that is~$\xi_\mu \in \LieH$ for all~$\mu$. This proves regularity of~$\LieH$.
\end{proof}

\begin{lemma}\label{e6-has-open-orbit-in-x1}
The action of~$\E_6$ on $\HV{\varpi_1} \times \HV{\varpi_1}$ has a dense orbit.
\end{lemma}
\begin{proof}
Let us start by calculating dimension of~$\HV{\varpi_1}$. Denote~$P\left( \varpi_1 \right)$ the set of weights that occur in the module~$V\left( \varpi_1 \right)$. Representation~$V\left( \varpi_1 \right)$ has the following property: if~$\xi \in \Lie \E_6$ and~$v \in V\left( \varpi_1 \right)$ are nonzero weight elements with such weights~$\mu$ and~$\nu$ that~$\mu + \nu \in P\left( \varpi_1 \right)$ then~$\xi \cdot v \neq 0$. In view of regularity of stabiliser of~$v_\lambda$ we conclude that this stabiliser is the direct sum of a subspace in~$\LieT$ of codimension~1 and weight spaces~$\LieG_\alpha$ with~$\varpi_1 + \alpha \not \in P\left( \varpi_1 \right)$. This reasoning shows that dimension of~$\HV{\varpi_1}$ equals~1 plus number of roots~$\alpha$ of~$\Lie \E_6$ such that~$\varpi_1 + \alpha \in P\left( \varpi_1 \right)$. Using this fact one easily finds that~$\dim \HV{\varpi_1} = 17$.

Remark that~$\HV{\varpi_1}$ has a vector of weight~$\varepsilon_6 - \varepsilon$. Indeed, the Weyl group of~$\E_6$ contains all permutations of~$\varepsilon_i$ and the mapping~$\varepsilon_i \mapsto \varepsilon_i, \varepsilon \mapsto -\varepsilon$ and these mappings can be used to obtain the necessary vector from the highest weight vector of~$V\left( \varpi_1 \right)$ by applying appropriate element of the Weyl group.

Consider a point of~$\HV{\varpi_1} \times \HV{\varpi_1}$ that has vector~$v$ of weight~$\varpi_1 = \varepsilon_1 + \varepsilon$ as its first component and vector~$w$ of weight~$\varepsilon_6 - \varepsilon$ as its second component. One can easily find the stabiliser of this point using the argument which has been employed for calculation of~$\dim \HV{\varpi_1}$. This argument shows that dimension of orbit of~$(v,w)$ is~34. Thus dimension of this orbit coincides with ~$\dim \HV{\varpi_1} \times \HV{\varpi_1}$. Therefore the orbit of~$(v,w)$ is dense.
\end{proof}

\begin{remark}
It is clear that the action described in the above lemma is not stable for it is not transitive. Indeed, all points in the described orbit have both components non-zero, so the point~$(0,0) \in \HV{\varpi_1} \times \HV{\varpi_1}$ is not contained in the orbit of~$(v,w)$.
\end{remark}

\subsection{Proof of Theorem~\ref{s_m-lower-bound}}

\begin{proof}
\textit{Case 1:~$G$ is~$\SL_n$}. If~$k < n$ then the action~$\SL_n : \left( \kk^n \right)^k$ is not transitive and has a dense orbit, hence~$s_m\left( \SL_n \right) \geq n = m\left( \SL_n \right)$.

\textit{Case 2:~$G$ is one of groups~$\Spin_{2n+1}$, $\Spin_{4n+4}$, $\Sp_{2n}$, $\G_2$, $\F_4$, $\E_7$, $\E_8$}. In all these cases the statement of the theorem is trivial because these groups have~$m(G) = 2$ and every action on a HV-variety is not stable.

\textit{Case 3:~$G$ is~$\Spin_{4n+2}$}. The triple of dominant weights~$\left( \varpi_{2n+1}, \varpi_{2n+1}, \varpi_{2n+1} \right)$ is invariant-free according to Lemma~\ref{half-spin-rep-is-inv-free}. In view of Theorem~\ref{inv-free-implies-instability} the action~$\Spin_{4n+2} : X^3$ with~$X = \HV{\varpi_{2n+1}}$ is not stable. Therefore~$s_m\left( \Spin_{4n+2} \right) \geq 4 = m\left( \Spin_{4n+2} \right)$.

\textit{Suppose~$G$ is~$\E_6$}. Lemma~\ref{e6-has-open-orbit-in-x1} gives an example of a diagonal action of~$\E_6$ with two copies which is not stable. Therefore~$s_m\left( \E_6 \right) \geq 3 = m\left( \E_6 \right)$.
\end{proof}

\subsection{More on bounds of stability indices}

Theorem~\ref{s_m-lower-bound} gives lower bound of~$s_m(G)$ and~$s_s(G)$. Lower bound of~$s_s(G)$ obtained in this way is not exact, that is, there are groups~$G$ that have~$s_s(G) > m(G)$. One example of such group is the symplectic group~$\Sp_{2m}$ because its standard representation requires passing to the diagonal action with~$2m$ copies in order to stabilise. On the other hand, Theorem~\ref{s_m-for-self-conjugate} shows that lower bound of~$s_m(G)$ is in many cases exact.

\prooflike{Proof of Theorem~\ref{s_m-for-self-conjugate}}
Let~$\theta$ be the Weyl involution of~$G$, that is, an involutive authomor\-phism of~$G$ that acts as inversion on some maximal torus in~$G$. It is well-known that for every linear representation~$R$ of group~$G$ its~$\theta$-twisting~$R \circ \theta$ is isomorphic to the conjugate representation~$R^*$. For an affine irreducible $G$-variety~$X$ set~$X^\theta$ to be~$X$ with~$\theta$-twisted action of~$G$.

Every affine~$G$-variety~$X$ admits an equivariant closed immersion~$X \hookrightarrow V$ into appropriate~$G$-module~$V$. Since all linear representations of~$G$ are self-conjugate we have an equivariant isomorphism~$\varphi : V \rightarrow V^\theta$ which can be used to construct isomorphism of~$X \times X$ with~$X \times X^\theta$. From this fact it follows that stability of action of~$G$ on~$X \times X$ is equivalent to stability of action~$G : X \times X^\theta$. The latter action is stable in view of~\cite[Proposition 1.6]{PanyushevSelfConjugate}, so we have~$s_m(G) = 2$.

To complete the proof we need to show that~$m(G) = 2$ for any semisimple group~$G$ that has only self-conjugate linear representations. This is obvious because for every~$G$-module~$V$ we have~$\left( V^{\otimes 2} \right)^G \isom$ $\left( V \otimes V^* \right)^G \isom$ $\left( \End(V) \right)^G$ and~$\End(V)$ contains a nonzero~$G$-invariant element, for example, the identity map~$\Id_V$. Thus~$2 \in M(G)$ and~$m(G) = 2$.
\endproof

\end{document}